\RequirePackage{fix-cm}
\documentclass[smallcondensed]{svjour3}     
\smartqed  
\usepackage{graphicx}
\usepackage{amsmath}
\usepackage{color}
\newtheorem{algorithm}[theorem]{Algorithm}
\begin{document}

\title{Parallel hybrid methods for generalized equilibrium problems and asymptotically strictly pseudocontractive mappings
}
\titlerunning{Parallel Methods for GEPs and FPPs}        
\author{Dang Van Hieu}
\authorrunning{D. V. Hieu} 
\institute{Dang Van Hieu \at
              Department of Mathematics, Vietnam National University, Ha Noi, Vietnam \\
              334 - Nguyen Trai Street, Ha Noi, Viet Nam\\
              Tel.: +84-979817776\\
              \email{dv.hieu83@gmail.com}           
}
\date{Received: date / Accepted: date}
\maketitle

\begin{abstract}
In this paper, we propose two novel parallel hybrid methods for finding a common element of the set of solutions of a finite family of generalized 
equilibrium problems for monotone bifunctions $\left\{f_i\right\}_{i=1}^N$ and $\alpha$ - inverse strongly monotone operators $\left\{A_i\right\}_{i=1}^N$ 
and the set of common fixed points of a finite family of  (asymptotically) $\kappa$- strictly pseudocontractive mappings $\left\{S_j\right\}_{j=1}^M$ in 
Hilbert spaces. The strong convergence theorems are established under the standard assumptions imposed on equilibrium bifunctions and operators. 
A numerical example is presented to illustrate the efficiency of the proposed parallel methods.
\vspace{5pt}
\end{abstract}
\section{Introduction}\label{intro}
Let $H$ be a real Hilbert space with the inner product $\left\langle .,. \right\rangle$ and the induced norm $||.||$ and $C$ be a nonempty closed convex 
subset of  $H$. Let $f: C\times C \to \Re$ be a bifunction and $A:C\to H$ be an $\alpha$ - inverse strongly monotone operator. The generalized equilibrium 
problem (GEP) for the bifunction $f$ and the monotone operator $A$ is defined as follows:
\begin{equation}\label{eq:GEP}
\mbox{Find}~ x\in C~ \mbox{such that:}~ f(x,y) +\left\langle Ax,y-x \right\rangle \ge 0,\quad \forall y\in C.
\end{equation}
The set of solutions of $(\ref{eq:GEP})$ is denoted by $GEP(f,A)$. The GEP $(\ref{eq:GEP})$ is very general in the sense that, it includes, as special cases, 
many mathematical models: optimization problems, saddle point problems, Nash equilirium point problems, fixed point problems, convex differentiable 
optimization problems, variational inequalities, complementarity problems, see e.g., \cite{BO1994,MO1992}. In recent years, some methods have been 
proposed for finding a point in the solution set $GEP(f,A)$, see \cite{KCQ2011,LZH2009,ZS2009}. We give two special cases for the GEP $(\ref{eq:GEP})$: \\
If $f=0$ then the GEP $(\ref{eq:GEP})$ becomes the variational inequality problem
\begin{equation}\label{eq:VIP}
\mbox{Find}~ x\in C~ \mbox{such that:}~ \left\langle Ax,y-x \right\rangle \ge 0,\quad \forall y\in C
\end{equation}
If $A=0$ then the GEP $(\ref{eq:GEP})$ becomes the equilibrium problem
\begin{equation}\label{eq:EP}
\mbox{Find}~ x\in C~ \mbox{such that:}~ f(x,y) \ge 0,\quad \forall y\in C.
\end{equation}
The sets of solutions of $(\ref{eq:VIP})$ and $(\ref{eq:EP})$ are denoted by $VI(A,C)$ and $EP(f,C)$, respectively.\\

Let $S:C\to C$ be a nonexpansive mapping with the set of fixed points $F(S)$. In 2003, Nakajo and Takahashi \cite{NT2003} 
introduced the following hybrid algorithm for finding a fixed point of the nonexpansive mapping $S$ in Hilbert spaces
$$
\left \{
\begin{array}{ll}
&x_0\in C_0:=C,\\
&y_n=\alpha_n x_n+(1-\alpha_n)Sx_n,\\
&C_{n}=\left\{z\in C:||y_n-z||\le||x_n-z||\right\},\\
&Q_n=\left\{z\in C:\left\langle x_n-z,x_0-x_n\right\rangle\ge 0\right\},\\
&x_{n+1}=P_{C_{n}\cap Q_n}x_0,n\ge 0.
\end{array}
\right.
$$
In 2010, Duan \cite{D2010} proposed a hybrid algorithm for finding a common element of the solution set $\left(\cap_{i=1}^M EP(f_i,C)\right)$ 
of equilibrium problems for the monotone bifunctions $\left\{f_i\right\}_{i=1}^M$ and the fixed point set $\left(\cap_{j=1}^N F(S_j)\right)$ of the strictly 
pseudocontractive mappings $\left\{S_j\right\}_{j=1}^N$ in Hilbert spaces which combines three methods including the proximal method \cite{CH2005}, 
the Mann iteration \cite{M1953} and the monotone hybrid (outer approximation) method. Precisely,
\begin{equation}\label{Duan2010}
\left \{
\begin{array}{ll}
&x_0\in C_0:=C,\\
&u_n= T_{r_{M,n}}^{f_M}\ldots T_{r_{1,n}}^{f_1}x_n,\\
&y_n=\alpha_nx_n+(1-\alpha_n)\left(\lambda_n u_n+(1-\lambda_n)S_{n{\rm (mod)}N}u_n\right),\\
&C_{n+1}=\left\{z\in C_n:||y_n-z||\le||x_n-z||\right\},\\
&x_{n+1}=P_{C_{n+1}}x_0,n\ge 0.
\end{array}
\right.
\end{equation}
Clearly, Duan's algorithm is inherently sequential. Thus it can be costly on a single processor when the numbers of bifunctions $M$ and of strictly pseudocontractive mappings $N$ are large.

Very recently, Anh and Hieu \cite{AH2014,AH2014b} have proposed the following parallel hybrid algorithm for finding a common fixed point of a 
finite family asymptotically quasi $\phi$ - nonexpansive mappings $\left\{T_i\right\}_{i=1}^N$ in uniformly smooth and uniformly convex Banach spaces
$$
\left \{
\begin{array}{ll}
&x_0\in C, ~C_0 := C,\\
&y_n^i=J^{-1}\left(\alpha_n Jx_n +(1-\alpha_n)JT^n_i x_n\right),i=1,2,\ldots,N,\\
&i_n=\arg\max_{1\le i\le N}\left\{\left\|y_n^i-x_n \right\|\right\}, \quad   \bar{y}_n := y_n^{i_n
},\\
&C_{n+1}:=\left\{v\in C_n:\phi(v,\bar{y}_n)\le \phi(v,x_n)+\epsilon_n\right\},\\
&x_{n+1}=\Pi_{C_{n+1}}x_0, n\ge 0,
\end{array}
\right.
$$
where $\epsilon_n:= (k_n-1)(\omega + ||x_n||)^2$, $\left\{\alpha_n\right\}\subset [0,1]$,  
$\lim\limits_{n\to\infty}\alpha_n =0$. Note that \cite{A1996,C1990} in Hilbert spaces the normalize duality mapping $J$ is the identity operator $I$, 
the Lyapunov funtional $\phi(x,y)=||x-y||^2,$ and the generalized projection $\Pi_C=P_C$. Arccoding to this algorithm, itermadiate approximations 
$y_n^i$ can be found in parallel, among all $y_n^i $ the furthest element from $x_n$, denoted by $\bar{y}_n$, is chosen. After that, 
based on $\bar{y}_n$, the closed convex set $C_{n+1}$ is constructed. Finally, the next approximation $x_{n+1}$ is defined as the projection 
of $x_0$ onto $C_{n+1}$. Some numerical experiments (see \cite{AC2013,AH2014,AH2014b}) have implied the efficency of this parallel algorithm. 
Moreover, it can be used to solve systems of monotone operator equations in Hilbert spaces or accretive operator equations in Banach spaces. 
Other parallel algorithms for finding a common solution of a finite family of accretive operator equations in Banach spaces can be found in 
\cite{ABH2014,AC2013,H2015}.
 
In this paper, motivated and inspired by above results we propose two new parallel hybrid algorithms for finding a common element of the set of 
solutions of a finite family of GEPs for bifunctions $\left\{f_i\right\}_{i=1}^N$ and operators $\left\{A_i\right\}_{i=1}^N$, and the set of common 
fixed points of finitely many (asymptotically) $\kappa$ - strictly pseudocontractive mappings $\left\{S_j\right\}_{j=1}^M$ in Hilbert spaces. The strong 
convergence theorems are proved under the widely used assumptions of equilibrium bifunctions and operators.

This paper is organized as follows: In Section $\ref{pre}$ we collect some definitions and primary results used in the next sections. In Section $\ref{main}$ 
we propose two parallel hybrid algorithms and prove their convergence. Finally, Section $\ref{numerical.example}$ presents a numerical example
to illustrate the efficiency of parallel computation of the proposed algorithms.
\section{Preliminaries}\label{pre}
In this section we recall some definitions and results for further use. Let $C$ be a nonempty closed and convex subset of a real Hilbert space $H$.
\begin{definition}\cite{BP1967,GK1972,Q1996}
A mapping $S:C\to C$ is said to be
\begin{itemize}
\item [$\rm i.$] \textit{nonexpansive} if  $||Sx-Sy||\le ||x-y||$ for all $x,y\in C;$
\item [$\rm ii.$] \textit{asymptotically nonexpansive} if there exists a sequence $\left\{k_n\right\}\subset [1;+\infty)$ with $k_n\to 1$ such that 
$$ ||S^n x-S^n y||\le k_n ||x-y|| ,\quad \forall x,y\in C, n\ge 1;$$
\item [$\rm iii.$] \textit{$\kappa$ - strictly pseudocontractive} if there exists a constant $\kappa\in [0;1)$ such that
$$ ||Sx-Sy||^2\le||x-y||^2+\kappa||(I-S)x-(I-S)y||^2,\quad \forall x,y\in C; $$
\item [$\rm iv.$] \textit{asymptotically $\kappa$ - strictly pseudocontractive} if there exist a constant $\kappa\in [0;1)$ and a sequence 
$\left\{k_n\right\}\subset [1;+\infty)$ with $k_n\to 1$ such that 
$$ ||S^nx-S^ny||^2\le k_n||x-y||^2+\kappa||(I-S^n)x-(I-S^n)y||^2,\forall x,y\in C, n\ge 1; $$
\end{itemize}
\end{definition}
The class of $\kappa$ - strictly pseudocontractive mappings was introduced by Browder and Petryshyn \cite{BP1967} in 1967. Clearly, each 
nonexpansive mapping is $0$ - strictly pseudocontractive. The class of asymptotically $\kappa$ - strictly pseudocontractive mappings \cite{Q1996} 
is a generalization of the one of $\kappa$ - strictly pseudocontractive mappings. A mentioned 
example in \cite{R1976} shows that the class of asymptotically $\kappa$ - strictly pseudocontractive mappings contains properly the one of 
$\kappa$ - strictly pseudocontractive mappings. We have the following result \cite{BP1967}.
\begin{lemma}\cite{SXY2009}\label{lem.demiclose}
Let H be a real Hilbert space and C be a nonempty closed convex subset of H. Let $S:C\to C$ be an asymptotically $\kappa$-strict pseudocontraction with 
the sequence $\left\{k_n\right\}\subset [1;\infty), k_n\to 1$. Then
\begin{enumerate}
\item [$\rm i.$] $F(S)$ is a closed convex subset of $H$.
\item [$\rm ii.$] $I-S$ is demiclosed, i.e., whenever $\left\{x_n\right\}$ is a sequence in $C$ weakly converging to some $x\in C$ and the sequence 
$\left\{(I-S)x_n\right\}$ strongly converges to some $y$, it follows that $(I-S)x=y$. 
\item [$\rm iii.$] $S$ is uniformly $L$ - Lipschitz continuous with the constant 
$$L=\sup\left\{\frac{\kappa+\sqrt{1+(1-\kappa)(k_n-1)}}{1+\kappa}:n\ge 1\right\},$$
i.e., $||S^nx-S^ny||\le L||x-y||$ for all $x,y\in C$ and $n\ge 1$.
\end{enumerate}
\end{lemma}
\begin{lemma}\cite{R1969}\label{lem.aux}
In a real Hilbert space H, the following equality holds
$$ ||ax+(1-a)y||^2= a||x||^2+(1-a)||y||^2-a(1-a)||x-y||^2,\quad \forall x,y\in H, a\in[0,1]. $$
\end{lemma}
\begin{definition}
A mapping $A:C\to H$ is said to be
\begin{itemize}
\item [$\rm i.$] \textit{monotone} if $\left\langle Ax-Ay,x-y\right\rangle\ge 0,\quad \forall x,y\in C$;
\item [$\rm ii.$] \textit{$\eta$ - strongly monotone} if there exists a constant $\eta>0$ such that 
$$
\left\langle Ax-Ay,x-y \right\rangle \ge \eta\left\|x-y\right\|^2,\quad \forall x,y\in C;
$$
\item [$\rm iii.$] \textit{$\alpha$ - inverse strongly monotone} if there exists a constant $\alpha>0$ such that
$$
\left\langle Ax-Ay,x-y \right\rangle \ge \alpha\left\|Ax-Ay\right\|^2,\quad \forall x,y\in C.
$$
\end{itemize}
\end{definition}
\begin{remark}\label{rem.strict_pseudocontraction}
A mapping $S:C\to C$ is $\kappa$ - strictly pseudocontractive iff $A=I-S$ is $\alpha$ - inverse strongly monotone ($0<\alpha<1$ and $\kappa=1-2\alpha$) 
and is pseudocontractive iff $A=I-S$ is monotone.
\end{remark}
\begin{remark}\label{remark1}
If $A$ is $\eta$ - strongly monotone and $L$ - Lipschitz continuous, i.e., $\left\|Ax-Ay\right\|$ $\le$ $L\left\|x-y\right\|$ for all $x,y\in C$ then $A$ is $\eta/L^2$ -  
inverse strongly monotone. If $T$ is nonexpansive then $A=I-T$ is $1/2$ - inverse strongly monotone and $VI(A,C)=F(T)$. 
\end{remark}
\begin{remark}\label{remark3}
If $A:C\to H$ is $\alpha$ - inverse strongly monotone then $A$ is $1/\alpha$ - Lipschitz continuous and $I-\lambda A$ is nonexpansive, where $\lambda\in (0,2\alpha)$.
\end{remark}
Indeed, from $\left\langle Ax-Ay,x-y \right\rangle \ge \alpha\left\|Ax-Ay\right\|^2$, we obtain 
$$||Ax-Ay||||x-y||\ge \alpha\left\|Ax-Ay\right\|^2.$$
This implies that $\left\|Ax-Ay\right\|\le 1/\alpha||x-y||$. Therefore $A$ is $1/\alpha$ - Lipschitz continuous.  Moreover,
\begin{eqnarray*}
||(I-\lambda A)x-(I-\lambda A)y||^2&=&||x-y||^2+\lambda^2||Ax-Ay||^2-2\lambda\left\langle Ax-Ay,x-y\right\rangle\\
&\le& ||x-y||^2+\lambda^2||Ax-Ay||^2-2\lambda \alpha ||Ax-Ay||^2\\
&=&||x-y||^2-\lambda(2\alpha-\lambda)||Ax-Ay||^2\\
&\le& ||x-y||^2.
\end{eqnarray*}
Hence, $I-\lambda A$ is nonexpansive.

For every $x\in H$, the element $P_C x$ is defined by
$$
P_C x=\arg\min\left\{\left\|y-x\right\|:y\in C\right\}.
$$
Since C is a nonempty closed convex subset of $H$, $P_C x$ exists and is unique. The mapping $P_C:H\to C$ is called the metric projection of $H$ 
onto $C$. It is also well - known that $P_C$ satisfies the following property
\begin{equation}\label{eq:FirmlyNonexpOfPC}
\left\langle P_C x-P_C y,x-y \right\rangle \ge \left\|P_C x-P_C y\right\|^2,
\end{equation}
which implies that $P_C$ is $1$ - inverse strongly monotone, and for all $x\in C, y\in H$,
\begin{equation}\label{eq:ProperOfPC}
\left\|x-P_C y\right\|^2+\left\|P_C y-y\right\|^2\le \left\|x-y\right\|^2.
\end{equation}
Moreover, $z=P_C x$ if and only if 
\begin{equation}\label{eq:EquivalentPC}
\left\langle x-z,z-y \right\rangle \ge 0,\quad \forall y\in C.
\end{equation}
For solving the GEP $(\ref{eq:GEP})$, we assume that the bifunction $f$ satisfies the following conditions:
\begin{itemize}
\item [$(A1)$] $f(x,x)=0$ for all $x\in C$;
\item [$(A2)$] $f$ is monotone, i.e., $f(x,y)+f(y,x)\le 0$ for all $x,y\in C$;
\item [$(A3)$] for all $x,y,z \in C$,
$$ \lim_{t\to 0}\sup f(tz+(1-t)x,y)\le f(x,y); $$
\item [$(A4)$] for each $x\in C$, the function $f(x,.)$ is convex and lower semicontinuos.
\end{itemize}
The following results concern with the befunction $f$.
\begin{lemma}\label{ExitenceN0} \cite{CH2005} Let $C$ be a
closed convex subset of a Hilbert space H, $f$ be a bifunction from $C\times C$ to
$\Re$ satisfying the conditions $(A1)$-$(A4)$ and let $r>0$,
$x\in H$. Then, there exists $z\in C$ such that
\begin{eqnarray*}
f(z,y)+\frac{1}{r}\langle y-z,z-x\rangle\geq0, \quad \forall y\in C.
\end{eqnarray*}
\end{lemma}
\begin{lemma}\label{Tr-ClosedConvex}\cite{CH2005} Let
$C$ be a closed convex subset of a Hilbert space $H$, $f$ be a bifunction from
$C\times C$ to $\Re$ satisfying the conditions $(A1)$-$(A4)$.
For all $r>0$ and $x\in H$, define the mapping
\begin{eqnarray*}
T_r^f x=\{z\in C:f(z,y)+\frac{1}{r}\langle y-z,z-x\rangle\geq0, \quad
\forall y\in C\}.
\end{eqnarray*}
Then the following hold:

{\rm (B1)} $T_r^f$ is single-valued;

{\rm (B2)} $T_r^f$ is a firmly nonexpansive, i.e., for
all $x, y\in H,$
\begin{eqnarray*}
||T_r^fx-T_r^fy||^2\leq\langle T_r^fx-T_r^fy,x-y\rangle;
\end{eqnarray*}

{\rm (B3)} $F(T_r^f)=EP(f,C);$

{\rm (B4)} $EP(f,C)$ is closed and convex.
\end{lemma}
\section{Main results}\label{main}
In this section, we propose two parallel hybrid algorithms  for finding a common element of the set of solutions of a finite family of GEPs for monotone 
bifunctions $\left\{f_i\right\}_{i=1}^N$ and $\alpha$ - inverse strongly monotone mappings $\left\{A_i\right\}_{i=1}^N$ and the set of common fixed 
points of a finite family of (asymptotically) $\kappa$ - strictly pseudocontractive mappings $\left\{S_j\right\}_{j=1}^M$ in Hilbert spaces. We assume 
that the mappings $\left\{A_i\right\}_{i=1}^N$ are inverse strongly monotone with the same constant $\alpha$ and $\left\{S_j\right\}_{j=1}^M$ are 
asymptotically $\kappa$ - strictly pseudocontractive mappings with the same sequence $\left\{k_n\right\}\subset [1;+\infty), k_n\to 1$ and constant 
$\kappa \in [0;1)$. Indeed, if $A_i$ is $\alpha_i$ - inverse strongly monotone, $A_i$ is $\alpha$ - inverse strongly monotone with 
$\alpha:=\min\left\{\alpha_i:i=1,\ldots,N\right\}$. Similarly, suppose that $S_j$ is asymptotically $\kappa_j$ - strictly pseudocontractive 
with the sequence $\left\{k_n^j\right\}\subset [1;+\infty), k_n^j\to 1$. Putting $k_n = \max\left\{k_n^j: j=1,\ldots,M\right\}$ and 
$\kappa:=\max\left\{\kappa_j:j=1,\ldots,M\right\}$. Then, $S_j$ is asymptotically $\kappa$ - strictly pseudocontractive with the sequence 
$\left\{k_n\right\}\subset [1;+\infty), k_n\to 1$. 

Moreover, we also assume that the solution set $$F:=\left(\cap_{i=1}^N GEP(f_i,A_i)\right)\bigcap\left(\cap_{j=1}^M F(S_j)\right)$$ is nonempty and bounded, i.e., there exists a positive real number $\omega$ such that $F\subset \Omega:=\left\{v\in H:||v||\le \omega\right\}$.
\begin{algorithm}\label{Algor.1}
\textbf{Initialization.} Choose $x_0\in C$ and set $n:=0$. The control parameter sequences $\left\{\alpha_k\right\},\left\{\beta_k\right\},\left\{r_k\right\}$ 
satisfy the following conditions.
\begin{itemize}
\item [$(a)$] $0<\alpha_k<1$, $\lim_{k\to\infty}\sup\alpha_k<1$;
\item [$(b)$] $\kappa\le\beta_k\le b<1$ for some $b\in (\kappa;1)$;
\item [$(c)$] $0<d\le r_k\le e<2\alpha$.
\end{itemize}
\textbf{Step 1.} Find intermediate approximations $y_n^i$ in parallel 
$$
y_n^i=T_{r_n}^{f_i}\left(x_n-r_nA_i(x_n)\right),~i=1,\ldots,N.
$$
\textbf{Step 2.} Choose the furthest element from $x_n$ among all $y_n^i$, i.e., 
$$ i_n = {\rm argmax}\{||y_n^i - x_n||: i =1,\ldots,N\},~\bar{y}_n:=y^{i_n}_n. $$
\textbf{Step 3.} Find intermediate approximations $z_n^j$ in parallel
$$ z_n^j=\alpha_n x_n+(1-\alpha_n)\left(\beta_n\bar{y}_n+(1-\beta_n) S_j^n \bar{y}_n\right),~ j=1,\ldots,M. $$
\textbf{Step 4.} Choose the furthest element from $x_n$ among all $z_n^j$, i.e., 
$$ j_n= {\rm argmax}\{||z_n^j - x_n||: j =1,\ldots,M\},~\bar{z}_n:=z^{j_n}_n. $$
\textbf{Step 5.} Construct the closed convex subset $C_{n+1}$ of $C$
$$
C_{n+1} = \{v \in C_n: ||\bar{z}_n - v||^2\leq ||x_n-v||^2+\epsilon_n\},
$$
where $\epsilon_n=(k_n-1)\left(||x_n||+\omega\right)^2$.\\
\textbf{Step 6.} The next approximation $x_{n+1}$ is defined as the projection of $x_0$ onto $C_{n+1}$, i.e.,
$$ x_{n+1}= P_{C_{n+1}}(x_0). $$
\textbf{Step 7.} Set $n:=n+1$ and go to \textbf{Step 1}.
\end{algorithm}
\begin{lemma}\label{lem.well-posed}
If Algorithm $\ref{Algor.1}$ reaches to the iteration $n\ge 0$ then $F\subset C_{n+1}$ and $x_{n+1}$ is well-defined.
\end{lemma}
\begin{proof}
From Lemmas $\ref{lem.demiclose}$ and $\ref{Tr-ClosedConvex}$, we see that $GEP(f_i,A),i=1,\ldots,N$ and $F(S_j),j=1,\ldots,M$ are closed 
convex subsets. Hence, $F$ is closed and convex. Moreover, we see that $C_0=C$ is closed and convex. Assume that $C_n$ is closed and convex 
for some $n\ge 0$. From the definition of $C_{n+1}$ we obtain 
$$C_{n+1}=C_n \cap \left\{v\in H: 2\left\langle v,x_n -\bar{z}_n\right\rangle \le ||x_n||^2-||\bar{z}_n||^2+\epsilon_n\right\}.$$
Hence, $C_{n+1}$ is closed and convex. By the induction, $C_{n}$ is closed and convex for all $n\ge 0$. Now, we show that $F\subset C_{n}$ 
for all $n\ge 0$. Putting $S_{j,\beta_n}=\beta_n I+(1-\beta_n) S_j^n $, hence $z_n^j=\alpha_n x_n+(1-\alpha_n)S_{j,\beta_n}\bar{y}_n$. 
From the nonexpansiveness of $T_{r_n}^{f_{i_n}}$, the inverse strongly monotonicity of $A_{i_n}$ and the hypothesis of $r_n$, we have, for each $u\in F$,
\begin{eqnarray}
||\bar{y}_n-u||^2&=&||T_{r_n}^{f_{i_n}}(x_n-r_n A_{i_n}x_n)-T_{r_n}^{f_{i_n}}(u-r_n A_{i_n}u)||^2\nonumber\\
&\le&||(x_n-r_n A_{i_n}x_n)-(u-r_n A_{i_n}u)||^2\nonumber\\
&=&||x_n-u||^2+r_n^2||A_{i_n}x_n-A_{i_n}u||^2-2r_n\left\langle A_{i_n}x_n-A_{i_n}u,x_n-u\right\rangle\nonumber\\
&\le&||x_n-u||^2+r_n^2||A_{i_n}x_n-A_{i_n}u||^2-2r_n\alpha||A_{i_n}x_n-A_{i_n}||^2\nonumber\\
&=&||x_n-u||^2-r_n(2\alpha-r_n)||A_{i_n}x_n-A_{i_n}u||^2\nonumber\\
&\le &||x_n-u||^2.\label{eq:1*}
\end{eqnarray}
Therefore, from the convexity of $||.||^2$ and the asymptotically $\kappa$ - strictly pseudocontractiveness of $S_{j_n}$,
\begin{eqnarray}
||\bar{z}_n-u||^2&=&||\alpha_n x_n+(1-\alpha_n)S_{j_n,\beta_n}\bar{y}_n-u||^2\nonumber\\
&\le&\alpha_n||x_n-u||^2+(1-\alpha_n)||S_{j_n,\beta_n}\bar{y}_n-u||^2\nonumber\\
&=&\alpha_n||x_n-u||^2+(1-\alpha_n)||\beta_n \bar{y}_n+(1-\beta_n) S_{j_n}^n\bar{y}_n-u||^2\nonumber\\
&=&\alpha_n||x_n-u||^2\nonumber\\
&&+(1-\alpha_n)\left(\beta_n ||\bar{y}_n-u||^2+(1-\beta_n)||S_{j_n}^n\bar{y}_n-S_{j_n}^nu||^2\right)\nonumber\\
&&-(1-\alpha_n)\beta_n(1-\beta_n)||(\bar{y}_n-u)-(S_{j_n}^n\bar{y}_n-S_{j_n}^hu)||^2\nonumber\\
&=&\alpha_n||x_n-u||^2+(1-\alpha_n)\left(\beta_n ||\bar{y}_n-u||^2+(1-\beta_n)k_n||\bar{y}_n-u||^2\right)\nonumber\\
&&+\kappa(1-\alpha_n)(1-\beta_n)||(I-S_{j_n}^n)\bar{y}_n-(I-S_{j_n}^n)u||^2\nonumber\\
&&-(1-\alpha_n)\beta_n(1-\beta_n)||(\bar{y}_n-u)-(S_{j_n}^n\bar{y}_n-S_{j_n}^hu)||^2\nonumber\\
&=&\alpha_n||x_n-u||^2+(1-\alpha_n)\left(\beta_n ||\bar{y}_n-u||^2+(1-\beta_n)k_n||\bar{y}_n-u||^2\right)\nonumber\\
&&-(\beta_n-\kappa)(1-\alpha_n)(1-\beta_n)||(I-S_{j_n}^n)\bar{y}_n-(I-S_{j_n}^n)u||^2\nonumber\\
&\le&\alpha_n||x_n-u||^2+(1-\alpha_n)\left(\beta_n ||\bar{y}_n-u||^2+(1-\beta_n)k_n||\bar{y}_n-u||^2\right)\nonumber\\
&=&\alpha_n||x_n-u||^2+(1-\alpha_n)||\bar{y}_n-u||^2+(1-\beta_n)(k_n-1)||\bar{y}_n-u||^2\nonumber\\
&\le&\alpha_n||x_n-u||^2+(1-\alpha_n)||x_n-u||^2+(k_n-1)||x_n-u||^2\nonumber\\
&\le&||x_n-u||^2+(k_n-1)\left(||x_n||+||u||\right)^2\nonumber\\
&\le&||x_n-u||^2+(k_n-1)\left(||x_n||+\omega\right)^2\nonumber\\
&=&||x_n-u||^2+\epsilon_n.\label{eq:2*}
\end{eqnarray}
This implies that $u\in C_{n+1}$ for all $u\in F$. Thus, by the induction $F\subset C_n$ for all $n\ge 0$. Since $F$ is nonempty, so is $C_{n+1}$. 
Hence $x_{n+1}$ is well-defined. The proof of Lemma $\ref{lem.well-posed}$ is complete.
\end{proof}
\begin{lemma}\label{lem.limits}
Suppose that $\left\{x_n\right\},\left\{y_n^i\right\}$ and $\left\{z_n^j\right\}$ are the sequences generated by Algorithm $\ref{Algor.1}$. 
Then, $\left\{x_n\right\}$ is a Cauchy sequence and there hold the following relations
$$
\lim_{n\to\infty}||x_n-y_n^i||=\lim_{n\to\infty}||x_n-z_n^j||=\lim_{n\to\infty}||x_n-S_jx_n||=0
$$
for all $i=1,\ldots,N$ and $j=1,\ldots,M$.
\end{lemma}
\begin{proof}
From the definition of $C_{n}$, we have $C_{n+1}\subset C_n$. Moreover, $x_{n+1}=P_{C_{n+1}}(x_0)\in C_{n+1}$. Thus $x_{n+1}\in C_n$. From $x_n=P_{C_n}x_0$ and the definition of $P_{C_n}$, we obtain 
$$
||x_{n}-x_0||\le ||x_{n+1}-x_0||.
$$
This implies that the sequence $\left\{||x_{n}-x_0||\right\}$ is nondecreasing. From $x_n=P_{C_n}x_0$, we also have $||x_n-x_0||\le ||u-x_0||$ for each 
$u\in F\subset C_n$. Thus, the sequence $\left\{||x_{n}-x_0||\right\}$ is bounded. Hence, there exists the limit of the sequence $\left\{||x_{n}-x_0||\right\}$. 
For all $m\ge n\ge 0$, we have $x_m\in C_n$. By $x_n=P_{C_n}x_0$ and the property $(\ref{eq:ProperOfPC})$ of the metric projection, we get
$$
||x_{m}-x_n||^2\le ||x_{m}-x_0||^2-||x_n-x_0||^2.
$$
Letting $m,n\to \infty$ in the last inequality, we obtain 
\begin{equation}\label{eq:2}
\lim_{m,n\to\infty}||x_{m}-x_n||=0.
\end{equation}
Therefore, $\left\{x_n\right\}$ is a Cauchy sequence. 
From $x_{n+1}\in C_{n+1}$ and the definition of $C_{n+1}$ we have 
\begin{equation}\label{eq:4}
||\bar{z}_n-x_{n+1}||^2\le ||x_n-x_{n+1}||^2+\epsilon_n.
\end{equation}
From the boundedness of $\left\{x_n\right\}$ and $k_n\to 1$, one has
\begin{equation}\label{eq:5}
\epsilon_n=(k_n-1)(||x_n||+\omega)^2\to 0
\end{equation}
as $n\to\infty$. Combining $(\ref{eq:2}),(\ref{eq:4}),(\ref{eq:5})$ we get
\begin{equation}\label{eq:6}
\lim_{n\to\infty}||\bar{z}_n-x_{n+1}||=0.
\end{equation}
From $(\ref{eq:2}),(\ref{eq:6})$ and $||\bar{z}_n-x_n||\le||\bar{z}_n-x_{n+1}||+||x_{n+1}-x_n||$ we get
\begin{equation}\label{eq:7}
\lim_{n\to\infty}||\bar{z}_n-x_{n}||=0.
\end{equation}
From the definition of $j_n$, we see that
\begin{equation}\label{eq:8}
\lim_{n\to\infty}||z_n^j-x_{n}||=0
\end{equation}
for all $j=1,\ldots,M$. 
From $z_n^j=\alpha_n x_n+(1-\alpha_n)S_{j,\beta_n}\bar{y}_n$ we obtain
$$ ||z_n^j-x_n||=(1-\alpha_n)||S_{j,\beta_n}\bar{y}_n-x_n||. $$
This equality together $(\ref{eq:8})$ and $\lim_{n\to\infty}\sup\alpha_n<1$ implies that
\begin{equation}\label{eq:10}
\lim_{n\to\infty}||S_{j,\beta_n}\bar{y}_n-x_n||=0.
\end{equation} 
For each $u\in F$, from the firmly nonexpansiveness of $T_{r_n}^{f_{i_n}}$, we have
\begin{eqnarray*}
2||\bar{y}_n-u||^2&=&2||T_{r_n}^{f_{i_n}}(x_n-r_n A_{i_n}x_n)-T_{r_n}^{f_{i_n}}(u-r_n A_{i_n}u)||^2\\
&\le&2\left\langle (x_n-r_n A_{i_n}x_n)-(u-r_n A_{i_n}u),\bar{y}_n-u\right\rangle\\
&=&||(x_n-r_n A_{i_n}x_n)-(u-r_n A_{i_n}u)||^2+||\bar{y}_n-u||^2\\
&&-||(x_n-r_n A_{i_n}x_n)-(u-r_n A_{i_n}u)-(\bar{y}_n-u)||^2\\
&\le&||x_n-u||^2+||\bar{y}_n-u||^2-||(x_n-\bar{y}_n)-r_n (A_{i_n}x_n-A_{i_n}u)||^2\\
&\le&||x_n-u||^2+||\bar{y}_n-u||^2-||x_n-\bar{y}_n||^2-r_n^2||A_{i_n}x_n-A_{i_n}u||^2\\
&&+2r_n\left\langle A_{i_n}x_n-A_{i_n}u,x_n-\bar{y}_n\right\rangle.
\end{eqnarray*}
Therefore
\begin{eqnarray}
||\bar{y}_n-u||^2&\le &||x_n-u||^2-||x_n-\bar{y}_n||^2-r_n^2||A_{i_n}x_n-A_{i_n}u||^2\nonumber\\
&&+2r_n\left\langle A_{i_n}x_n-A_{i_n}u,x_n-\bar{y}_n\right\rangle\nonumber\\
&\le&||x_n-u||^2-||x_n-\bar{y}_n||^2-r_n^2||A_{i_n}x_n-A_{i_n}u||^2\nonumber\\
&&+2r_n||A_{i_n}x_n-A_{i_n}u||||x_n-\bar{y}_n||\nonumber\\
&\le&||x_n-u||^2-||x_n-\bar{y}_n||^2+2r_n||A_{i_n}x_n-A_{i_n}u||||x_n-\bar{y}_n||.\label{eq:12}
\end{eqnarray}
By arguing similarly as in $(\ref{eq:2*})$, we obtain
\begin{equation}\label{eq:12*}
||\bar{z}_n-u||^2\le\alpha_n||x_n-u||^2+\epsilon_n+(1-\alpha_n)||\bar{y}_n-u||^2.
\end{equation}
The last inequality together with the relation $(\ref{eq:1*})$ one has
\begin{eqnarray*}
||\bar{z}_n-u||^2&\le&\alpha_n||x_n-u||^2+\epsilon_n+(1-\alpha_n)||\bar{y}_n-u||^2\\
&\le&\alpha_n||x_n-u||^2+\epsilon_n+(1-\alpha_n)||x_n-u||^2\\
&&-(1-\alpha_n)r_n(2\alpha-r_n)||A_{i_n}x_n-A_{i_n}u||^2\\ 
&=& ||x_n-u||^2+\epsilon_n-(1-\alpha_n)r_n(2\alpha-r_n)||A_{i_n}x_n-A_{i_n}u||^2.
\end{eqnarray*}
Therefore 
$$(1-\alpha_n)r_n(2\alpha-r_n)||A_{i_n}x_n-A_{i_n}u||^2\le||x_n-u||^2-||\bar{z}_n-u||^2+\epsilon_n.$$
This implies that
\begin{equation}\label{eq:13}
(1-\alpha_n)r_n(2\alpha-r_n)||A_{i_n}x_n-A_{i_n}u||^2\le||x_n-\bar{z}_n||\left(||x_n-u||+||\bar{z}_n-u||\right)+\epsilon_n.
\end{equation}
From $(\ref{eq:5}),(\ref{eq:7}),(\ref{eq:13})$ and the boundedness of $\left\{x_n\right\},\left\{\bar{z}_n\right\}$ we obtain
\begin{equation}\label{eq:14}
\lim_{n\to\infty}||A_{i_n}x_n-A_{i_n}u||=0.
\end{equation}
From $(\ref{eq:12})$ and $(\ref{eq:12*})$, we get
\begin{eqnarray}
(1-\alpha_n)||x_n-\bar{y}_n||^2&\le&||x_n-u||^2-||\bar{z}_n-u||^2+\epsilon_n\nonumber\\
&&+(1-\alpha_n)2r_n||A_{i_n}x_n-A_{i_n}u||||x_n-\bar{y}_n||\nonumber\\
&\le&||x_n-\bar{z}_n||\left(||x_n-u||+||\bar{z}_n-u||\right)+\epsilon_n\nonumber\\
&&+(1-\alpha_n)2r_n||A_{i_n}x_n-A_{i_n}u||||x_n-\bar{y}_n||.\label{eq:15}
\end{eqnarray}
Combining $(\ref{eq:5}),(\ref{eq:7}),(\ref{eq:14}),(\ref{eq:15})$, we have
\begin{equation}\label{eq:16}
\lim_{n\to\infty}||x_n-\bar{y}_n||=0.
\end{equation}
From the definition of $i_n$, we conclude that
$$
\lim_{n\to\infty}||x_n-y_n^i||=0,\quad i=1,\ldots,N,
$$
From $(\ref{eq:10})$ and $(\ref{eq:16})$, one has
$$
\lim_{n\to\infty}||S_{j,\beta_n}\bar{y}_n-\bar{y}_n||=0
$$
or
$$
\lim_{n\to\infty}(1-\beta_n)||S_{j}^n\bar{y}_n-\bar{y}_n||=0.
$$
Since $\beta_n\le b<1$, 
\begin{equation}\label{eq:19}
\lim_{n\to\infty}||S_{j}^n\bar{y}_n-\bar{y}_n||=0.
\end{equation}
Thus, it follows from the uniformly $L$ - Lipschitz continuity of $S_j$ that
\begin{eqnarray*}
||S_j^nx_n-x_n||&\le&||S_j^nx_n-S_j^n\bar{y}_n||+||S_j^n\bar{y}_n-\bar{y}_n||+||\bar{y}_n-x_n||\\ 
&\le& L||x_n-\bar{y}_n||+||S_j^n\bar{y}_n-\bar{y}_n||+||\bar{y}_n-x_n||\\ 
&=&(L+1)||x_n-\bar{y}_n||+||S_j^n\bar{y}_n-\bar{y}_n||,
\end{eqnarray*}
where $L$ is defined as in Lemma \ref{lem.demiclose}. This together with $(\ref{eq:16})$ and $(\ref{eq:19})$ implies that
\begin{equation}\label{eq:20}
\lim_{n\to\infty}||S_{j}^nx_n-x_n||=0.
\end{equation}
From the triangle inequality and the uniformly $L$ - Lipschitz continuity of $S_j$,
\begin{eqnarray*}
||S_{j}x_n-x_n||&\le&||S_{j}x_n-S_j^{n+1}x_n||+||S_{j}^{n+1}x_n-S_j^{n+1}x_{n+1}||\\
&&+||S_{j}^{n+1}x_{n+1}-x_{n+1}||+||x_{n+1}-x_n||\\ 
&=&L||x_n-S_j^n x_n|| +(L+1)||x_{n}-x_{n+1}||+||S_{j}^{n+1}x_{n+1}-x_{n+1}||,
\end{eqnarray*}
which, from the relations (\ref{eq:2}) and $(\ref{eq:20})$, implies that
\begin{equation}\label{eq:21}
\lim_{n\to\infty}||S_{j}x_n-x_n||=0.
\end{equation}
The proof of Lemma $\ref{lem.limits}$ is complete.
\end{proof}
\begin{lemma}\label{lem.pinF}
Suppose that $p$ is a limit point of $\left\{x_n\right\}$ then $p\in F$.
\end{lemma}
\begin{proof}
By Lemma $\ref{lem.limits}$, $\left\{x_n\right\}$ is a Cauchy sequence in $C$. Since $C$ is closed, $x_n\to p\in C$. From Lemma $\ref{lem.demiclose}$ we see that $S_j$ is Lipschitz continuous and so continuous. Thus, the equality $(\ref{eq:21})$ gives $S_j p=p$. Hence $p\in F(S_j)$ for all $j=1,\ldots,M$.

Next, we show that $p\in \cap_{i=1}^N GEP(f_i,A_i)$. Note that $\lim_{n\to\infty}\left\|y_n^i-x_n\right\|=0$. This together with $r_n\ge d>0$ implies that
\begin{equation}\label{eq:22}
\lim_{n\to\infty}\frac{\left\|y_n^i-x_n\right\|}{r_n}=0.
\end{equation}
Moreover, since $A_i$ is $\alpha$-inverse strongly monotone, $A_i$ is Lipschitz continuous. Hence 
\begin{equation}\label{eq:22*}
\lim_{n\to\infty}||A_i y_n^i-A_ix_n ||=0.
\end{equation}
We have $y_n^i=T_{r_n}^{f_i}(x_n-r_n A_i x_n)$, i.e.,
\begin{equation}\label{eq:23}
f_i(y_n^i,y)+\left\langle A_i x_n, y-y_n^i\right\rangle+\frac{1}{r_n}\left\langle y-y_n^i,y_n^i-x_n\right\rangle \ge 0, \forall y\in C.
\end{equation}
From $(\ref{eq:23})$ and $(A2)$, we get
\begin{equation}\label{eq:24}
\left\langle A_i x_n, y-y_n^i\right\rangle+\frac{1}{r_n}\left\langle y-y_n^i,y_n^i-x_n\right\rangle \ge -f_i(y_n^i,y)\ge f_i(y,y_n^i)\quad \forall y\in C.
\end{equation}
For $0<t\le 1$ and $y\in C$, putting $y_t=ty+(1-t)p$. Since $y\in C$ and $p\in C$, $y_t \in C$. Hence, for each $t\in (0,1]$, from $(A3)$ and $(\ref{eq:24})$, we have that
\begin{eqnarray*}
\left\langle y_t-y_n^i, A_i y_t\right\rangle &\ge&\left\langle y_t-y_n^i, A_i y_t\right\rangle-\left\langle A_i x_n, y_t-y_n^i\right\rangle\\
&&-\frac{1}{r_n}\left\langle y_t-y_n^i,y_n^i-x_n\right\rangle+f_i(y_t,y_n^i)\\ 
&\ge& \left\langle y_t-y_n^i, A_i y_t-A_i y_n^i\right\rangle+\left\langle A_i y_n^i-A_i x_n, y_t-y_n^i\right\rangle\\
&&-\frac{1}{r_n}\left\langle y_t-y_n^i,y_n^i-x_n\right\rangle+f_i(y_t,y_n^i)\\ 
&\ge&\left\langle A_i y_n^i-A_i x_n, y_t-y_n^i\right\rangle-\frac{1}{r_n}\left\langle y_t-y_n^i,y_n^i-x_n\right\rangle+f_i(y_t,y_n^i).
\end{eqnarray*}
Letting $n\to\infty$ in the last inequality, from $(\ref{eq:22}),(\ref{eq:22*})$ and the hypothesis $(A4)$, we have
\begin{equation}\label{eq:25}
\left\langle y_t-p, A_i y_t\right\rangle \ge f_i(y_t,p).
\end{equation}
By $(A1),(A4)$ and $(\ref{eq:25})$, one has
\begin{eqnarray*}
0&=&f_i(y_t,y_t)\\ 
&=&f_i(y_t,ty+(1-t)p) \\
&\le& tf_i(y_t,y)+(1-t)f_i(y_t,p)\\
&\le& tf_i(y_t,y)+(1-t)\left\langle y_t-p, A_i y_t\right\rangle\\
&=&tf_i(y_t,y)+(1-t)t\left\langle y-p, A_i y_t\right\rangle.
\end{eqnarray*}
Dividing both sides of the last inequality by $t>0$, we obtain 
$$ f_i(y_t,y)+(1-t)\left\langle y-p, A_i y_t\right\rangle\ge 0,\, \forall y\in C. $$
Taking $t\to 0^+$ in the last inequality, from $(A3)$, we get $f_i(p,y)+\left\langle y-p, A_i p\right\rangle\ge 0$ for all $y\in C$ and $1\le i\le N$, i.e, $p\in \cap_{i=1}^N GEP(f_i,A_i)$. Therefore, $p\in F$. The proof of Lemma $\ref{lem.pinF}$ is complete.
\end{proof}
\begin{theorem}\label{theo1} Let $C$ be a nonempty closed convex subset of a Hilbert space $H$. Suppose that $\left\{f_i\right\}^N_{i=1}$ 
is a finite family of bifunctions satisfying the conditions $(A1)-(A4)$; $\left\{A_i\right\}^N_{i=1}$ is a finite family of $\alpha$ - inverse strongly 
monotone mappings; $\left\{S_j\right\}^M_{j=1}$ is a finite family of asymptotically $\kappa$ - strictly pseudocontractive mappings with the 
sequence $\left\{k_n\right\}\subset [1;+\infty),k_n\to 1$. Moreover, suppose that the solution set $F$ is nonempty and bounded. Then the sequences 
$\left\{x_n\right\},\left\{y_n^i\right\}$ and $\left\{z_n^j\right\}$ generated by Algorithm $\ref{Algor.1}$ converge strongly to $P_F x_0$.
\end{theorem}
\begin{proof}
By Lemmas $\ref{lem.limits}$ and $\ref{lem.pinF}$, the sequences $\left\{x_n\right\},\left\{y_n^i\right\}$ and $\left\{z_n^j\right\}$ 
converge strongly to $p\in F$. Now, we show that $x_n\to x^\dagger:=P_{F}x_0$. Indeed, from the proof of Lemma $\ref{lem.limits}$ and $x^\dagger \in F$ we get
$$
\left\|x_n-x_0\right\|\le \left\|x^\dagger-x_0\right\|.
$$
By the continuity of $\left\|.\right\|$ we have
$$
\left\|p-x_0\right\|=\lim_{n\to\infty}\left\|x_n-x_0\right\|\le \left\|x^\dagger-x_0\right\|.
$$
By the definition of $x^\dagger$, $p=x^\dagger$. The proof of Theorem $\ref{theo1}$ is complete.
\end{proof}
\begin{corollary}\label{cor1}
Assume that $\left\{f_i\right\}_{i=1}^N,\left\{A_k\right\}_{k=1}^K,\left\{S_j\right\}_{j=1}^M, \left\{\alpha_n\right\},\left\{\beta_n\right\},\left\{r_n\right\}$ 
satisfy all conditions in Theorem $\ref{theo1}$. In addition the set 
$F$ $=$ $\left(\cap_{i=1}^N EP(f_i,C)\right)$ $\bigcap$ $\left(\cap_{k=1}^K VI(A_k,C)\right)$ $\bigcap$ $\left(\cap_{j=1}^M F(S_j)\right)$ is nonempty 
and bounded. Let $\left\{x_n\right\}$ be the sequence generated by the following manner:
$$
\left\{
\begin{array}{ll}
&x_0\in C_0:=C,\\
&y_{n}^i=T_{r_n}^{f_i}x_n, i=1,\ldots, N,\\
&y_{n}^{N+k}=P_C(x_n-r_n A_kx_n), k=1,\ldots, K,\\
&i_n=\arg\max\left\{||y_{n}^i-x_n||:i=1,\ldots, N+K\right\}, \bar{y}_n=y_n^{i_n},\\
&z_n^j=\alpha_n x_n+(1-\alpha_n)\left(\beta_n\bar{y}_n+(1-\beta_n) S_j^n \bar{y}_n\right), j=1,\ldots,M,\\
&j_n= {\rm argmax}\{||z_n^j - x_n||: j =1,\ldots,M\},\bar{z}_n:=z^{j_n}_n,\\
&C_{n+1} = \{v \in C_n: ||\bar{z}_n - v||^2\leq ||x_n-v||^2+\epsilon_n\},\\
&x_{n+1}= P_{C_{n+1}}(x_0), n\ge 0.
\end{array}
\right.
$$
Then the sequence $\left\{x_n\right\}$ converges strongly to $P_Fx_0$.
\end{corollary}
\begin{proof}
Putting $F_i=f_i$ for $i=1,\ldots, N$ and $F_i=0$ for $i=N+1,\ldots, N+K$. Putting $B_i=0$ for $i=1,\ldots, N$ and $B_{N+k}=A_k$ for $k=1,\ldots,K$. 
Now, we condider the following $N+K$ generalized equilibrium problems
$$
\mbox{Find}\, x\in C\, \mbox{such that:}\, F_i(x,y) +\left\langle B_ix,y-x \right\rangle \ge 0,\quad \forall y\in C,
$$
where $i=1,\ldots,N+K$. Application of Algorithm $\ref{Algor.1}$ with the finite family of above GEPs leads to
$$
y_n^i=T_{r_n}^{F_i}(x_n-r_n B_ix_n),\,i=1,\ldots,N+K.
$$
For each $i=1,\ldots,N$, since $B_i=0$ and $F_i=f_i$, $y_n^i=T_{r_n}^{f_i}(x_n)$. Moreover, for each $i=N+1,N+2,\ldots,N+K$, since $F_i=0$ and $B_i=A_i$, $y_n^i=P_C(x_n-r_n A_ix_n).$ Thus, Corollary $\ref{cor1}$ is followed directly from Theorem $\ref{theo1}$.
\end{proof}
\begin{corollary}\label{cor2}
Assume that $\left\{A_k\right\}_{i=1}^N,\left\{S_j\right\}_{j=1}^M, \left\{\alpha_n\right\},\left\{\beta_n\right\},\left\{r_n\right\}$ satisfy all conditions in 
Theorem $\ref{theo1}$. In addition the solution set $F$ $=$ $\left(\cap_{i=1}^N VI(A_i,C)\right)$ $\bigcap$ $\left(\cap_{j=1}^M F(S_j)\right)$ is 
nonempty and bounded. Let $\left\{x_n\right\}$ be the sequence generated by the following manner:
$$
\left\{
\begin{array}{ll}
&x_0\in C_0:=C,\\
&y_{n}^{i}=P_C(x_n-r_n A_ix_n), i=1,\ldots, N,\\
&i_n=\arg\max\left\{||y_{n}^i-x_n||:i=1,\ldots, N\right\}, \bar{y}_n=y_n^{i_n},\\
&z_n^j=\alpha_n x_n+(1-\alpha_n)\left(\beta_n\bar{y}_n+(1-\beta_n) S_j^n \bar{y}_n\right), j=1,\ldots,M,\\
&j_n= {\rm argmax}\{||z_n^j - x_n||: j =1,\ldots,M\},\bar{z}_n:=z^{j_n}_n,\\
&C_{n+1} = \{v \in C_n: ||\bar{z}_n - v||^2\leq ||x_n-v||^2+\epsilon_n\},\\
&x_{n+1}= P_{C_{n+1}}(x_0), n\ge 0.
\end{array}
\right.
$$
Then the sequence $\left\{x_n\right\}$ converges strongly to $P_Fx_0$.
\end{corollary}
\begin{proof}
Corollary $\ref{cor2}$ is followed from Theorem $\ref{theo1}$ with $f_i(x,y)=0$ for all $i=1,\ldots,N$.
\end{proof}
\begin{corollary}\label{cor3}
Let $f$ be a bifunction from $C\times C$ to $\Re$ satisfying all conditions $\left(A1\right)-\left(A4\right)$, $A$ be an $\alpha$- inverse strongly monotone 
mapping, and $S$ be an asymptotically $\kappa$ - strict pseudocontraction mapping. The control parameter sequences 
$\left\{\alpha_n\right\},\left\{\beta_n\right\},\left\{r_n\right\}$ satisfy all conditions in Theorem $\ref{theo1}$. Moreover, assume that the solution 
set $F=GEP(f,A)\cap F(S)$ is nonempty and bounded. Let $\left\{x_n\right\}$ be the sequence generated by the following manner
$$
\left\{
\begin{array}{ll}
&x_0\in C_0:=C,\\
&y_n\,\, {\rm such \, that}\,\,f(y_n,y)+\left\langle Ax_n, y-y_n\right\rangle+\frac{1}{r_n}\left\langle y-y_n,y_n-x_n\right\rangle \ge 0, \forall y\in C,\\
&z_n=\alpha_n x_n+(1-\alpha_n)\left(\beta_ny_n+(1-\beta_n) S^ny_n\right),\\
&C_{n+1} = \{v \in C_n: ||z_n - v||^2\leq ||x_n-v||^2+\epsilon_n\},\\
&x_{n+1}= P_{C_{n+1}}(x_0), n\ge 0.
\end{array}
\right.
$$
Then the sequence $\left\{x_n\right\}$ converges strongly to $P_Fx_0$.
\end{corollary}
\begin{proof}
Corollary $\ref{cor3}$ is followed from Theorem $\ref{theo1}$ with $f_i(x,y)=f(x,y)$, $A_i=A$, $S_j=S$ for all $i=1,\ldots,N$ and $j=1,\ldots,M$.
\end{proof}
\begin{corollary}\label{cor4}
Assume that $\left\{f_i\right\}_{i=1}^N,\left\{A_i\right\}_{i=1}^N, \left\{\alpha_n\right\}, \left\{r_n\right\}$ satisfy all conditions in Theorem 
$\ref{theo1}$ and $\left\{S_j\right\}_{j=1}^M$ is a finite family of asymptotically nonexpansive mappings with the same sequence 
$\left\{k_n\right\}\subset [1,\infty), k_n \to 1$ as $n\to\infty$. Moreover, assume that the solution set 
$F=\left(\cap_{i=1}^N EP(f_i,A_i)\right)\bigcap\left(\cap_{j=1}^M F(S_j)\right)$ is nonempty and bounded. 
Let $\left\{x_n\right\}$ be the sequence generated by the following manner
$$
\left\{
\begin{array}{ll}
&x_0\in C_0:=C,\\
&y_n^i\,\, {\rm such \, that}\,\,f_i(y_n^i,y)+\left\langle Ax_n, y-y_n^i\right\rangle+\frac{1}{r_n}\left\langle y-y_n^i,y_n^i-x_n\right\rangle \ge 0, \forall y\in C,\\
&i_n=\arg\max\left\{||y_{n}^i-x_n||:i=1,\ldots, N\right\}; \bar{y}_n=y_n^{i_n},\\
&z_n^j=\alpha_n x_n+(1-\alpha_n)S_j^n \bar{y}_n, j=1,\ldots,M,\\
&j_n= {\rm argmax}\{||z_n^j - x_n||: j =1,\ldots,M\},\bar{z}_n:=z^{j_n}_n,\\
&C_{n+1} = \{v \in C_n: ||\bar{z}_n - v||^2\leq ||x_n-v||^2+\epsilon_n\},\\
&x_{n+1}= P_{C_{n+1}}(x_0), n\ge 0,
\end{array}
\right.
$$
where $\epsilon_n=(k_n^2-1)(||x_n||+\omega)^2$. Then the sequence $\left\{x_n\right\}$ converges strongly to $P_Fx_0$.
\end{corollary}
\begin{proof}
Since $S_j$ is an asymptotically nonexpansive mapping with the sequence $\left\{k_n\right\}\subset [1,\infty), k_n \to 1$ as $n\to\infty$, $S_j$ is an 
asymptotically $0$ - strictly pseudocontraction mapping with the sequence $\left\{k_n^2\right\}\subset [1,\infty), k_n^2 \to 1$ as $n\to\infty$. 
Using Theorem $\ref{theo1}$ with $\kappa=\beta_n=0$, we obtain the desired conclusion.
\end{proof}
For a finite family of $\kappa$ - strictly pseudocontractive mappings, the assumption of the boundedness of the set $F$ is redundant. 
We have the following Algorithm.
\begin{algorithm}\label{Algor.2}
\textbf{Initialization.} Choose $x_0\in C$ and set $n:=0$, $C_0=C$. The control parameter sequences $\left\{\alpha_k\right\}$, $\left\{\beta_k\right\}$, $\left\{r_k\right\}$ 
satisfy the following conditions
\begin{itemize}
\item [$(a)$] $0<\alpha_k<1$, $\lim_{k\to\infty}\sup\alpha_k<1$;
\item [$(b)$] $\kappa\le\beta_k\le b<1$ for some $b\in (\kappa;1)$;
\item [$(c)$] $0<d\le r_k\le e<2\alpha$.
\end{itemize}
\textbf{Step 1.} Find intermediate approximations $y_n^i$ in parallel
$$
y_n^i=T_{r_n}^{f_i}\left(x_n-r_nA_i(x_n)\right),~i=1,\ldots,N.
$$
\textbf{Step 2.} Choose the furthest element from $x_n$ among all $y_n^i$, i.e., 
$$ i_n = {\rm argmax}\{||y_n^i - x_n||: i =1,\ldots,N\},\bar{y}_n:=y^{i_n}_n. $$
\textbf{Step 3.} Find intermediate approximations $z_n^j$ in parallel
$$ z_n^j=\alpha_n x_n+(1-\alpha_n)\left(\beta_n\bar{y}_n+(1-\beta_n) S_j \bar{y}_n\right), j=1,\ldots,M. $$
\textbf{Step 4.} Choose the furthest element from $x_n$ among all $z_n^j$, i.e., 
$$ j_n= {\rm argmax}\{||z_n^j - x_n||: j =1,\ldots,M\},\bar{z}_n:=z^{j_n}_n. $$
\textbf{Step 5.} Construct the closed convex subset $C_{n+1}$ of C
$$
C_{n+1} = \{v \in C_n: ||\bar{z}_n - v||\leq ||x_n-v||\}.
$$
\textbf{Step 6.} The next approximation $x_{n+1}$ is defined as the projection of $x_0$ onto $C_{n+1}$, i.e.,
$$ x_{n+1}= P_{C_{n+1}}(x_0). $$
\textbf{Step 7.} Set $n:=n+1$ and go to \textbf{Step 1}.
\end{algorithm}
\begin{theorem}\label{theo2} Let $C$ be a nonempty closed convex subset of a Hilbert space $H$. Suppose that $\left\{f_i\right\}^N_{i=1}$ is a 
finite family of bifunctions satisfying the conditions $(A1)-(A4)$; $\left\{A_i\right\}^N_{i=1}$ is a finite family of $\alpha$ - inverse strongly monotone 
mappings; $\left\{S_j\right\}^M_{j=1}$ is a finite family of $\kappa$ - strictly pseudocontractive mappings. Moreover, suppose that the solution set 
$F$ is nonempty. Then the sequences $\left\{x_n\right\},\left\{y_n^i\right\}$ and $\left\{z_n^j\right\}$ generated by Algorithm $\ref{Algor.2}$ 
converge strongly to $P_F x_0$.
\end{theorem}
\begin{proof}
Since $S_j$ is a $\kappa$ - strictly pseudocontractive mapping, $S_j$ is an asymptotically $\kappa$ - strictly pseudocontractive mapping with the 
sequence $k_n=1$ for all $n\ge 1$. Putting $\epsilon_n=0$, by the same arguments as in the proof of Theorem $\ref{theo1}$ we obtain $F, C_n$ 
are closed convex subsets of $C$. Moreover, $F\subset C_n$ for all $n\ge 0$ and 
$$
\lim_{n\to\infty}x_n=\lim_{n\to\infty}y_n^i=\lim_{n\to\infty}z_n^j= p\in C.
$$
for all $i=1,\ldots,N$ and $j=1,\ldots,M$. Since $S_j$ is an asymptotically $\kappa$ - strictly pseudocontractive mapping 
with the sequence $k_n=1$, from Lemma $\ref{lem.demiclose}$, we see that $S_j$ is $1$ - Lipschitz continuous. By arguing 
similarly to $(\ref{eq:21})$, we obtain
$$
\lim_{n\to\infty}||S_j x_n-x_n||=0,\quad j=1,\ldots,M.
$$
The rest of the proof  of Theorem $\ref{theo2}$ is similar to the one of Theorem $\ref{theo1}$.
\end{proof}
\begin{corollary}\label{cor5}
Suppose that $\left\{f_i\right\}^N_{i=1}$, $\left\{S_j\right\}^M_{j=1}$ and control parameter sequences $\left\{\alpha_n\right\}$, $\left\{\beta_n\right\}$, 
$\left\{r_n\right\}$ satisfy all conditions in Theorem $\ref{theo2}$. In addition, the solution set $F=\left(\cap_{i=1}^N EP(f_i,C)\right)$ $\bigcap$ 
$\left(\cap_{j=1}^M F(S_j)\right)$ is nonempty. Let $\left\{x_n\right\}$ be the sequence generated by the following manner:
\begin{equation}\label{algor.cor5}
\left\{
\begin{array}{ll}
&x_0\in C_0:=C,\\
&y_n^i=T_{r_n}^{f_i}x_n, i=1,\ldots,N,\\
&i_n=\arg\max\left\{||y_{n}^i-x_n||:i=1,\ldots, N\right\}, \bar{y}_n=y_n^{i_n},\\
&z_n^j=\alpha_n x_n+(1-\alpha_n)\left(\beta_n \bar{y}_n+(1-\beta_n) S_j \bar{y}_n\right), j=1,\ldots,M,\\
&j_n= {\rm argmax}\{||z_n^j - x_n||: j =1,\ldots,M\},\bar{z}_n:=z^{j_n}_n,\\
&C_{n+1} = \{v \in C_n: ||\bar{z}_n - v||\leq ||x_n-v||\},\\
&x_{n+1}= P_{C_{n+1}}(x_0), n\ge 0.
\end{array}
\right.
\end{equation} 
Then, the sequence $\left\{x_n\right\}$ converges strongly to $P_F x_0$.
\end{corollary}
\begin{proof}
Theorem $\ref{theo2}$ with $A_i=0$ ensures that the sequence $\left\{x_n\right\}$ converges strongly to $P_F x_0$.
\end{proof}
\begin{remark}
Corollary $\ref{cor5}$ gives a parallel hybrid algorithm which improves announced results in \cite{D2010} .
\end{remark}
\section{Numerical example}\label{numerical.example}
Let $H$ be  the set of real numbers with the strandard inner product $\left\langle x,y\right\rangle=xy$ and the induced norm 
$||x||=|x|$ for all $x,y\in H$. We condider a finite family of mappings $\left\{S_j\right\}_{j=1}^M$ on $C=[-1,1]$ as follows: $S_j(x)=x$ for all $-1\le x< 0$ 
and $S_j(x)=x-c_j x^2$ for all $0\le x\le 1$, where $1<c_j< 2$, $j=1,\ldots,M$. Putting $A_j(x)=0$ for all $-1\le x< 0$ and $A_j(x)=c_j x^2$ for all 
$0\le x\le 1$. A straightforward computation implies that $A_j$ is $c_j/4$ - inverse strongly monotone. By Remark $\ref{rem.strict_pseudocontraction}$, 
$S_j=I-A_j$ is $\kappa_j$ - strictly pseudocontractive mapping with $\kappa_j=1-c_j/2$. Besides, for each fixed real number $\frac{2}{c_j}-1<\epsilon<1$, 
we have
$$ |S_j(\epsilon)-S_j(1)|=|1-\epsilon||c_j(1+\epsilon)-1|>|1-\epsilon|. $$
Therefore, $S_j$ is not nonexpansive. It is clear that $\cap_{j=1}^M F(S_j)=[-1,0]$.

Let $f_i:C\times C\to \Re$ be bifunctions defined as follows:
$$
f_i(x,y)=B_i(x)(y-x), \quad \forall x,y\in C, i=1,\ldots,N,
$$
where $B_i(x)=0$ for all $x\in [-1,\xi_i)$ and $B_i(x)=\tan(x-\xi_i)-x+\xi_i$ for all $x\in [\xi_i,1]$, and $-1<\xi_1<\ldots<\xi_N<1$. All conditions $(A1), (A3), (A4)$ are satisfied automatically. Since $B_i(x)$ is nondecreasing, $f_i(x,y)$ is monotone. Finally,
$$ f_i(x,y)=B_i(x)(y-x)\ge 0,\quad \forall y\in [-1,1] $$
if and only if $-1\le x\le \xi_i$, i.e., $EP(f_i,C)=[-1,\xi_i]$. Therefore, $\cap_{i=1}^N EP(f_i,C)=[0,\xi_1]$ and
$$ F=\left(\cap_{j=1}^M F(S_j)\right) \bigcap \left(\cap_{i=1}^N EP(f_i,C)\right)=[0,\xi_1].$$
By Algorithm $(\ref{algor.cor5})$ we have
\begin{equation}\label{algor.cor5.example}
\left\{
\begin{array}{ll}
&y_n^i =T_{r_n}^{f_i}x_n , i=1,\ldots,N,\\
&i_n=\arg\max\left\{||y_{n}^i-x_n||:i=1,\ldots, N\right\}, \bar{y}_n=y_n^{i_n},\\
&z_n^j=\alpha_n x_n+(1-\alpha_n)\left(\beta_n \bar{y}_n+(1-\beta_n) S_j \bar{y}_n\right), j=1,\ldots,M,\\
&j_n= {\rm argmax}\{||z_n^j - x_n||: j =1,\ldots,M\},\bar{z}_n:=z^{j_n}_n,\\
&C_{n+1} = \{v \in C_n: ||\bar{z}_n - v||\leq ||x_n-v||\},\\
&x_{n+1}= P_{C_{n+1}}(x_0), n\ge 0.
\end{array}
\right.
\end{equation} 
We choose $x_0=1$ and $r_n=1$ for all $n\ge 0$. By Corollary $\ref{cor5}$ then $x_n \to x^\dagger:=P_F x_0= \xi_1$. We see that $y_n^i=T_{r_n}^{f_i}x_n$ is equivalent to the following problem: Find $y_n^i\in [-1,1]$ such that 
$$B_i(y_n^i)(y-y_n^i)+(y-y_n^i)(y_n^i-x_n)\ge 0, \quad \forall y\in [-1,1]$$
or $$ (y-y_n^i)(B_i(y_n^i)+y_n^i-x_n)\ge 0, \quad \forall y\in [-1,1].$$
The last inequality is equivalent to the equation $B_i(y_n^i)+y_n^i=x_n$. Therefore, if $-1\le x_n<\xi_i$ then $y_n^i=x_n$. Otherwise, 
if $\xi_i\le x_n\le 1$ then $y_n^i=\xi_i+\arctan(x_n-\xi_i)$. Among all $y_n^i$, the furthest element from $x_n$, denoted by $\bar{y}_n$, is chosen. 
From $(\ref{algor.cor5.example})$ we have
$$ z_n^j=\alpha_n x_n+(1-\alpha_n)\left(\beta_n \bar{y}_n+(1-\beta_n) S_j \bar{y}_n\right). $$
If  $\bar{y}_n<0$ then $z_n^j=\alpha_n x_n+(1-\alpha_n)\bar{y}_n$. Otherwise, if $\bar{y}_n\ge 0$ then 
$z_n^j=\alpha_n x_n+(1-\alpha_n)\left(\bar{y}_n-c_j(1-\beta_n) \bar{y}_n^2\right)$. Thus, we can choose the furthest element from $x_n$ among all 
$z_n^j$ and denote by $\bar{z}_n$. We have $C_{n+1} = \{v \in C_n: |\bar{z}_n - v|\leq |x_n-v|\}$. By the induction, we can show that 
$$ C_{n+1}=[-1,\frac{x_n+\bar{z}_n}{2}].$$
Therefore, $x_{n+1}=P_{C_{n+1}}x_0=\frac{x_n+\bar{z}_n}{2}$. We obtain the following algorithm.\\
\textbf{Initialization.} Choose $x_0=1$, $\rm TOL=10^{-l}~(l=3,4,6)$, $ N =2\times 10^6, M=3\times 10^6,\xi_i=-1+\frac{2i}{N+1}, i=1,\ldots,N, 
c_j=2-\frac{j}{M+1},j=1,\ldots,M, \kappa=\max\left\{k_j:j=1,\ldots,M\right\}=\frac{M}{2(M+1)}, r_k=1, \beta_k=\frac{M}{2(M+1)}, 
\alpha_k=\frac{1}{k+1}, k=1,2,\ldots$. Set $ n:=0.$\\
\textbf{Step 1.} Find intermediate approximations $y_n^i$ in parallel  ($i=1,\ldots,N$)
$$
y_n^i=
\left\{
\begin{array}{ll}
&x_n\quad \mbox{ if }\quad -1\le x_n<\xi_i,\\ 
&\xi_i+\arctan(x_n-\xi_i)\quad \mbox{ if }\quad \xi_i\le x_n\le 1.
\end{array}
\right.
$$
\textbf{Step 2.} Choose the furthest element from $x_n$ among all $y_n^i$ ($i=1,\ldots,N$)
$$
i_n=\arg\max\left\{||y_{n}^i-x_n||:i=1,\ldots, N\right\}; \bar{y}_n=y_n^{i_n}.
$$
\textbf{Step 3.} Find intermediate approximations $z_n^j$ in parallel ($j=1,\ldots,M$)
$$
z_n^j=
\left\{
\begin{array}{ll}
&\alpha_n x_n+(1-\alpha_n)\bar{y}_n\quad \mbox{ if }\quad -1<\bar{y}_n<0,\\ 
&\alpha_n x_n+(1-\alpha_n)\left(\bar{y}_n-c_j(1-\beta_n) \bar{y}_n^2\right)\quad \mbox{ if }\quad 0\le \bar{y}_n\le 1.
\end{array}
\right.
$$
\textbf{Step 4.} Choose the furthest element from $x_n$ among all $z_n^j$ ($j=1,\ldots,M$)
$$
j_n=\arg\max\left\{||z_{n}^j-x_n||:j=1,\ldots, M\right\}; \bar{z}_n=z_n^{j_n}.
$$
\textbf{Step 5.} Compute $x_{n+1}=\frac{x_n+\bar{z}_n}{2}$.\\
\textbf{Step 6.} If $|x_{n+1}-\xi_1|\le \rm TOL$ then stop. Otherwise, set $n:=n+1$ and go to \textbf{Step 1.}

\end{document}